\documentclass[a4paper,10pt,reqno, english]{amsart}

\usepackage{amsmath,amssymb,amscd,amsthm,amsfonts}
\usepackage{graphicx,subfigure}
\usepackage{hyperref}
\usepackage{dsfont}
\usepackage[nobysame, alphabetic]{amsrefs}
\usepackage{tikz}
\usepackage[capitalise]{cleveref}
\usepackage{mathrsfs}

\newtheorem{theorem}{Theorem}
\newtheorem{lemma}{Lemma}

\newtheorem{corollary}{Corollary}

\newtheorem{remark}{Remark}
\newtheorem{proposition}{Proposition}

\def\rr{\mathds{R}}

\DeclareMathOperator{\relint}{relint}
\DeclareMathOperator{\conv}{conv}
\DeclareMathOperator{\aff}{aff}

\title{Blocking codimension-one simplices on the moment curve}

\hypersetup{
  pdftitle={Blocking codimension-one simplices on the moment curve},
  pdfauthor={Pablo Sober\'on},
  hidelinks
}

\author[Sober\'on]{Pablo Sober\'on}
\address{Baruch College, City University of New York, One Bernard Baruch Way, New York, NY 10010, United States}
\email{psoberon@gc.cuny.edu}

\thanks{The research of P. Sober\'on is supported by NSF CAREER grant DMS-2237324 and a PSC-CUNY Trad B award.}



\begin{document}

\begin{abstract}
We study $b_d(n)$, the minimum number of points needed to meet the relative interior of every
$(d-1)$-simplex spanned by an $n$-point set in general position in $\rr^d$.  In the plane, this is the parameter from the
Blocking Conjecture.  We improve the best known general planar lower bound to
\[
 b_2(n)\ge \frac{41}{13}n-O\left(\frac{n}{\log n}\right).
\]

For $n$ points on the moment curve in even dimension $2r$, we prove that at least $\frac{1}{r!}n^r\log n-O_r(n^r)$ points are needed to pierce the relative interior of all its codimension-one simplices, which exceeds the number of codimension-one faces in a triangulation by a $\log n$ factor.  For equally spaced points on the moment curve in odd dimensions, we construct an optimal blocking set whose size equals the maximum number of codimension-one faces in a triangulation.
\end{abstract}

\maketitle

\section{Introduction}

Let $P$ be a set of $n$ points in the plane, no three of which are collinear.
A set $Q\subset \rr^2\setminus P$ \emph{blocks} $P$ if every segment with
endpoints in $P$ contains a point of $Q$ in its relative interior.  We call the elements of $Q$ \textit{blockers}.  We denote
by $b(P)$ the smallest size of such a set and define
\[
 b(n)=\min\{b(P): |P|=n \text{ and }P\text{ is in general position}\}.
\]
The \emph{Blocking Conjecture}, as called by P\'or and Wood
\cite{Por2010}, states that $b(n)$ is superlinear in $n$.  In other words, $
 \lim_{n \to \infty}{b(n)}/{n}=\infty$.  Pach's construction of point sets with few
midpoints gives
\(
 b(n)\le n\exp\bigl(O(\sqrt{\log n})\bigr)
\)
\cite{Pach2003}.  In the other direction, Matou\v{s}ek proved $b(n)\ge2n-3$
and a lower bound of order $n\log n$ when the points are in convex position
\cite{Matousek2009}.  Dumitrescu, Pach, and T\'oth subsequently obtained
\[
 b(n)\ge\left(\frac{25}{8}-o(1)\right)n
\]
\cite{Dumitrescu2009}.  The problem is discussed in later work on visibility and obstacle numbers \cites{Mukkamala2012,GhoshGoswami2013, Balko2025}.

Our first result is a small improvement of the general lower bound.

\begin{theorem}\label{thm:planar-main}
 We have
 \(
 b(n)\ge \frac{41}{13}n-O\left(\frac{n}{\log n}\right)\).
\end{theorem}

The proof, given in \cref{sec:plane}, follows the same principle as the
$25/8$ bound.  A triangulation gives a lower bound for the number of points needed to pierce the relative interior of each segment, and
an empty pentagon forces us to use one additional point.  The earlier proof used
Harborth's theorem that every $10$-point set contains an empty pentagon
\cite{Harborth1978}.  We use Scheucher's computer-assisted theorem that every
$15$-point set contains two empty pentagons with disjoint interiors
\cite{Scheucher2020}.

We also consider a direct higher-dimensional extension.  Let $P\subset\rr^d$
be in general position, meaning that every set of at most $d+1$ points is
affinely independent.  Define
\[
 b_d(P)=\min\left\{|Q|:
 Q\subset\rr^d\setminus P,\quad
 Q\cap\relint\conv(A)\neq\emptyset
 \text{ for every }A\in\binom{P}{d}\right\},
\]
and let $b_d(n)$ be the minimum of $b_d(P)$ over all such $n$-point sets.
Therefore $b_2(P)=b(P)$ and $b_2(n) = b(n)$.  This parameter is the point-stabbing case for
$(d-1)$-simplices in the terminology of Cano, Hurtado, and Urrutia
\cite{Cano2014}.  Their extremal function maximizes over point
sets, while we take the minimum as in the planar Blocking Conjecture.  It is
also related in spirit, but not in quantifiers, to selection problems asking
for one flat that stabs many simplices \cite{Bukh2010}.

Every triangulation $T$ of $\conv (P)$ with vertex set equal to $P$ gives the immediate lower bound \(b_d(P)\ge f_{d-1}(T)\), where $f_{d-1}(T)$ denotes the number of $(d-1)$-faces of $T$.  We can define \( \tau_d(P)=\max_T f_{d-1}(T)\)
and ask when $b_d(P)$ is substantially larger than this simple lower bound.  This will be helpful to distinguish the behavior between $d=2r$ and $d=2r-1$ in our results.

The main higher-dimensional result gives a logarithmic excess for points on
the moment curve in every even dimension.  The moment curve $\gamma_d \subset \rr^d$ is the image of the function \(\gamma_d(t)=(t,t^2,\dots,t^d)\).  The moment curve is particularly useful in the study of polytopes.  Placing points on the moment curve generates cyclic polytopes, and intuitively corresponds to one of the strongest higher-dimensional generalizations of the notion of ``points in convex position'' in dimension two \cites{Gale1963, Ziegler1995, Grunbaum2003}.

\begin{theorem}\label{thm:even-main}
 Let $r$ be a positive integer, let $d = 2r, n \ge 2r$, and let $P$ be a set of $n$ points on the moment curve $\gamma_{d}(t)$.
 Then
 \(
 b_{2r}(P)\ge
 \sum_{\ell=1}^{n-2r+1}\frac{1}{\ell}
 \binom{n-\ell-r+1}{r}\).
 In particular,
 \(
 b_{2r}(P)\ge \frac{1}{r!}n^r\log n-O_r(n^r)\).
\end{theorem}

The proof in \cref{sec:even} is an extension of Matou\v{s}ek's argument for a
convex polygon.  This is essentially a charging and a double-counting argument.  The key step in the proof is that we assign weights only to a selected family of $(d-1)$-simplices spanned by $P$, which allows us to simplify the analysis significantly, yet still preserve the $n^r\log n$ leading term.

The boundary of a cyclic $2r$-polytope has $\Theta(n^r)$ facets, and the Upper
Bound Theorem implies $\tau_{2r}(P)=\Theta_r(n^r)$; see, for example,
\cites{Ziegler1995,Stanley1975}.  This means that \cref{thm:even-main} implies
\[
 \frac{b_{2r}(P)}{\tau_{2r}(P)}=\Omega_r(\log n).
\]
This contrasts with the following odd-dimensional construction.

\begin{theorem}\label{thm:odd-main}
 Let $r\ge2$, let $n\ge2r$, and let \( P=\{\gamma_{2r-1}(1),\gamma_{2r-1}(2),\dots,
       \gamma_{2r-1}(n)\}\).
 Then
 \[
 b_{2r-1}(P)=
 (n-2r+2)\binom{n-r}{r-1}=\tau_{2r-1}(P).
 \]
\end{theorem}

Therefore, in odd dimension the $\log n$ excess on the number of blockers disappears.  We establish the difference in behavior based on parity of the dimension only
for moment-curve configurations; it is unclear how much of it persists
for arbitrary point sets.

In terms of lower bounds, a known bound on the number of simplices in triangulations gives the following linear lower bound for $b_d(n)$ for $d \ge 3$.

\begin{proposition}\label{thm:universal-intro}
 Let $d\ge 3$ be an integer.  Then, for $n$ sufficiently large we have ${b_d(n) \ge \binom{d+1}{2}n + \Omega_d(\log n)}$.
\end{proposition}

We prove \cref{thm:planar-main} first in \cref{sec:plane} and \cref{thm:universal-intro} in
\cref{sec:framework}.  The even and odd-dimensional moment-curve results are
proved in \cref{sec:even} and \cref{sec:odd}, respectively.

\section{Linear lower bounds}

\subsection{The planar Blocking Conjecture}\label{sec:plane}

A \emph{$5$-hole} of a planar point set $P$ is a set of five points in convex
position whose convex hull contains no other point of $P$.  We call two
$5$-holes \emph{compatible} if the interiors of their convex hulls are
disjoint.

The proof has one simple idea.  The edges of a triangulation require distinct
blockers.  Inside an empty pentagon, however, the five diagonals require at
least three blockers, while a triangulation uses only two diagonals.  Each
compatible pentagon therefore contributes one blocker beyond the
triangulation count.


\begin{lemma}\label{lem:pentagon-surplus}
 Let $P$ be a set of $n$ points in general position in the plane, and let $h$
 be the number of vertices of $\conv(P)$.  If $P$ contains $k$ pairwise
 compatible $5$-holes, then
 \[
 b(P)\ge 3n-h-3+k.
 \]
\end{lemma}

\begin{proof}
 Suppose we have a set $S$ of blockers.  Since the pentagons have disjoint interiors, their boundary edges form a
 plane straight-line graph.  Extend this graph to a triangulation $T$ of $P$.
 The triangulation has $3n-h-3$ edges, whose relative interiors are pairwise
 disjoint.  Choose one blocker from $S$ on each edge of $T$; these blockers are all
 distinct.

 Fix one of the $5$-holes.  The restriction of $T$ to the pentagon uses two
 internal diagonals.  Its five diagonals require at least three blockers, as every blocker is contained in at most two diagonals, all
 in the interior of the pentagon.  Therefore, the pentagon contains a blocker
 not among the two chosen for its triangulation diagonals.  The pentagon
 interiors are disjoint, so these additional $k$ blockers are distinct.
\end{proof}

We use Scheucher's theorem to obtain a large number of pairwise compatible $5$-holes.  See \cref{fig:pentagons}.  The regions are obtained recursively by cutting off $15$ points
with a rotating support line.  Consecutive regions share the two points on
the cutting line, so each step consumes only $13$ new points.

\begin{figure}
    \centering
    \includegraphics[]{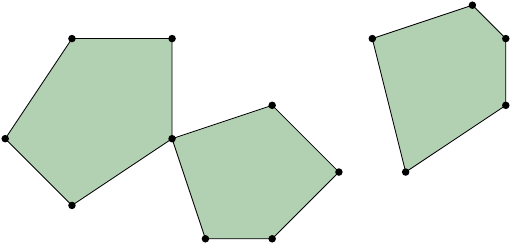}
    \caption{Three compatible $5$-holes.  Notice that the pentagons do not need to be disjoint, only their interiors have to be disjoint.  To find two disjoint $5$-holes, Scheucher shows that at most $17$ points are needed.  To find two compatible $5$-holes, $15$ points are enough.}
    \label{fig:pentagons}
\end{figure}

\begin{lemma}\label{lem:decomposition}
 Every set $P$ of $n$ points in general position contains \( s=\left\lfloor\frac{n-2}{13}\right\rfloor \) subsets $P_1,\dots,P_s$, each of size $15$, such that \( P\cap\conv(P_i)=P_i \)
 for every $i$, and the interiors of the convex hulls $\conv(P_i)$ are
 pairwise disjoint.
\end{lemma}

\begin{proof}
 We describe one recursive step.  Let $S\subseteq P$ have at least $15$
 points and satisfy $P\cap\conv(S)=S$.  Choose an edge $pq$ of $\conv(S)$ and
 rotate its supporting line around $p$ until the
 closed halfplane that initially contained only $p$ and $q$ contains exactly
 $15$ points of $S$, one of which is $q$.  General position implies that the final line goes through
 exactly two points of $S$: one is $p$ and the other some other point $q'$.

 Let $A$ be the $15$ points in the cut-off closed halfplane and let $S'$ be
 the points in the opposite closed halfplane.  The two sets share $p$ and $q'$, and therefore
 \[
 |S'|=|S|-13.
 \]
 Their convex hulls have disjoint interiors.  Moreover,
 $P\cap\conv(A)=A$ and $P\cap\conv(S')=S'$, because both convex hulls are
 contained in $\conv(S)$ and in their respective closed halfplanes.

 Let $P_1=A$ and continue with $S'$ with the same process to construct $P_2,\dots,P_s$.  All later regions lie on the opposite
 side of the line used to cut off $A$, so the convex hulls of $P_1$ and any of $P_2,\dots,P_s$ will have
 pairwise disjoint interiors.  The process continues precisely while the
 remaining set has at least $15$ points, giving
 $\lfloor(n-2)/13\rfloor$ total sets.
\end{proof}

\begin{corollary}\label{cor:pentagons}
 Every set $P$ of $n$ points in general position contains at least \( 2\cdot\left\lfloor\frac{n-2}{13}\right\rfloor \) pairwise compatible $5$-holes.
\end{corollary}

\begin{proof}
 Apply Scheucher's theorem \cite{Scheucher2020} to each set from
 \cref{lem:decomposition}.  The condition
 $P\cap\conv(P_i)=P_i$ makes the resulting pentagons empty with respect to all
 of $P$, and pentagons obtained in different regions have disjoint interiors.
\end{proof}

Combining \cref{lem:pentagon-surplus,cor:pentagons}, we obtain
\begin{equation}\label{eq:finite-planar}
 b(P)\ge 3n-h-3+2\left\lfloor\frac{n-2}{13}\right\rfloor
 \ge \frac{41}{13}n-h-\frac{69}{13}.
\end{equation}

To remove the term $h$, we use Matou\v{s}ek's convex-position estimate.  If
$X$ consists of $h$ points in convex position, then
\begin{equation}\label{eq:convex-bound}
 b(X)\ge h\log(h/2).
\end{equation}
This follows immediately from 
\cite{Matousek2009}.  All logarithms are natural.

\begin{proof}[Proof of \cref{thm:planar-main}]
 Let $P$ be an $n$-point set in general position and let $h$ be the number of
 vertices of $\conv(P)$.

 If $h\le 4n/\log n$, then \eqref{eq:finite-planar} gives
 \[
 b(P)\ge
 \frac{41}{13}n-\frac{4n}{\log n}-\frac{69}{13}.
 \]

 Suppose that $h>4n/\log n$.  Every blocking set for $P$ also blocks the
 segments determined by the vertices of the convex hull.  By \eqref{eq:convex-bound},
 \[
 b(P)\ge h\log(h/2)
 >\frac{4n}{\log n}\log\left(\frac{2n}{\log n}\right)
 =\frac{4 n}{\log n}(\log n + \log 2 - \log (\log (n)))=(4-o(1))n.
 \]
 This is larger than $(41/13)n$ for all sufficiently large $n$.  The two
 cases prove the theorem.
\end{proof}

\begin{remark}\label{rem:pentagon-parameter}
 If every $m$-point set in general position contains $t$
 pairwise compatible $5$-holes, the recursive argument proves
 \[
 b(n)\ge\left(3+\frac{t}{m-2}-o(1)\right)n.
 \]
 In other words, this method improves $41/13$ when
 \(
 \frac{t}{m-2}>\frac{2}{13}.
 \)
 Any result obtained with this argument will still be a linear lower bound.
\end{remark}

There have been recent improvements on the number of $5$-holes in a set of $n$ points in the plane.  The current best lower bound is $\Omega (n^{20/11})$ by Astudillo-Marb{\'a}n and Sol\'e-Pi \cite{astudillo2026}.

\subsection{Lower bounds in higher dimensions}\label{sec:framework}

We now turn to $P\subset\rr^d$ in general position.  As a first simple observation, for every triangulation $T$ of $P$,
\( b_d(P)\ge f_{d-1}(T)\).

Aichholzer, F\'abila-Monroy, Hackl, Huemer, and Urrutia showed that for $d \ge 3$, and $n$ sufficiently large, every set of $n$ points in $\rr^d$ in general position has a triangulation with at least $dn + \Omega_d(\log n)$ maximal simplices \cite{Aichholzer2014}.  We use their result to prove \cref{thm:universal-intro}.

\begin{proof}[Proof of \cref{thm:universal-intro}]
    Consider a triangulation with at least $dn + \Omega_d(\log n)$ maximal $d$-simplices.  Every simplex has $d+1$ facets, and every facet is counted at most twice.  Therefore,
    \[
    b_d(n) \ge \frac{d+1}{2}\Big( dn + \Omega_d(\log n)\Big) = \binom{d+1}{2}n + \Omega_d(\log n).
    \]
\end{proof}

\section{Even-dimensional moment curves}\label{sec:even}

The main idea to prove \cref{thm:even-main} is to extend Matou\v{s}ek's planar argument.  We first give a sketch of Matou\v{s}ek's idea, since it will be helpful for the arguments that follow.  Given $n$ points in the plane in convex position, we start by measuring the ``length'' of a chord by the minimum number of edges of the convex hull that it leaves on one side.  A chord that is part of the boundary of the convex hull has length $1$, a chord that skips one vertex has length $2$, and so on.  

Matou\v{s}ek assigns weight $1/\ell$ to a chord of length $\ell$.  Among all chords passing through one blocker, choose a
shortest one.  Every other chord uses a different vertex on its shorter arc,
so there are at most $\ell-1$ other chords and the total weight received by the
blocker is at most one.  The rest of the argument is to show that the total weight of the system is $\Omega(n \log n)$ via a simple harmonic sum, which implies the lower bound on the number of blockers needed.

We use the same charging argument in dimension $2r$.  We will assign weights only to a certain kind of codimension-one simplices.  The selected simplices need to satisfy two conditions:
\begin{itemize}
    \item there must be a simple way to assign weights to simplices so that the total weight of all simplices that contain a blocker is $O(1)$, and
    \item there must be sufficiently many simplices with weights assigned so that the total weight of the system is $\Omega (n^r \log n)$.
\end{itemize}

We describe the construction of this special family of simplices and their assigned weights.  The order of the vertices given by their placement on the moment curve is essential.

We use the following standard alternation property of polynomial systems (see, e.g., \cites{Karlin1963, Ziegler1968}).  We include an elementary proof for completeness.

\begin{lemma}\label{lem:sign-changes}
 Let $x_1<\cdots<x_s$ and let $c_1,\dots,c_s$ be nonzero real numbers such
 that
 \[
 \sum_{i=1}^s c_i x_i^j=0
 \qquad\text{for }j=0,1,\dots,d.
 \]
 Then the sign sequence
 $\operatorname{sgn}(c_1),\dots,\operatorname{sgn}(c_s)$ has at least $d+1$
 sign changes.
\end{lemma}

\begin{proof}
 Suppose for a contradiction that there are only $m\le d$ sign changes.  Choose one real number
 strictly between the two consecutive support points at each sign change.
 After changing an overall sign if necessary, the polynomial $f$ whose roots
 are these $m$ numbers has the same sign as $c_i$ at every $x_i$.  This implies that
 \[
 \sum_{i=1}^s c_i f(x_i)>0.
 \]
 On the other hand, $\deg f\le d$, so expanding $f$ shows that the sum should be zero, giving us the desired contradiction.
\end{proof}

Let $t_1<t_2<\dots < t_n$ be real numbers such that $P = \{\gamma_{2r}(t_1),\dots,\gamma_{2r}(t_n)\}$.  This gives us an order of the vertices.  We consider
$2r$-subsets of the form
\begin{equation}\label{eq:near-facet}
 A=\{i,j,k_1,k_1+1,\dots,k_{r-1},k_{r-1}+1\},
\end{equation}
where
\[
 i<j<k_1<k_1+1<\cdots<k_{r-1}<k_{r-1}+1.
\]
We call $(i,j)\subset \rr$ the \emph{free interval} of $A$.  For $r=1$, these are simply
all pairs $\{i,j\}$.

We say that two intervals $(i,j)$ and $(i',j')$ in $\rr^1$ properly cross if either $i < i' < j < j'$ or $i' < i < j'< j$.  The next lemma is the key geometric component of the proof.  Intuitively, if $P$ is a subset of the moment curve and $p \in \conv P$, two convex combinations with positive coefficients that generate $p$
alternate.  The consecutive pairs in \eqref{eq:near-facet} can contribute at
most one vertex each to an alternating subsequence.  Both free endpoints are
therefore forced to appear, and the two free intervals must properly cross.

\begin{lemma}\label{lem:crossing-intervals}
 Let $A$ and $B$ be distinct sets of the form \eqref{eq:near-facet} such that
 \[
 \relint\conv\{\gamma_{2r}(t_a):a\in A\}
 \cap
 \relint\conv\{\gamma_{2r}(t_b):b\in B\}
 \neq\emptyset.
 \]
Then their free intervals $(i,j)$ and $(i',j')$ properly cross.
\end{lemma}

For example, in dimension four the tetrahedra we take are those with vertices of the form $\{\gamma_4(t_i), \gamma_4(t_j), \gamma_4(t_k), \gamma_4(t_{k+1})\}$ with $i < j < k$.

\begin{proof}
 Let $q$ lie in the two relative interiors.  There are strictly positive
 coefficients $\alpha_a$ and $\beta_b$, each family summing to one, such that
 \[
 \sum_{a\in A}\alpha_a
 (1,t_a,t_a^2,\dots,t_a^{2r})
 =
 \sum_{b\in B}\beta_b
 (1,t_b,t_b^2,\dots,t_b^{2r}).
 \]
 This means that
 \[
 0  = \sum_{a\in A}\alpha_a
 (1,t_a,t_a^2,\dots,t_a^{2r})
 -
 \sum_{b\in B}\beta_b
 (1,t_b,t_b^2,\dots,t_b^{2r})
 \]
 Combine coefficients at common support points and delete zero coefficients.
 The resulting nonzero coefficients annihilate all polynomials of degree at
 most $2r$.  By \cref{lem:sign-changes}, its signs have at least $2r+1$
 changes, so its support contains an alternating subsequence of length
 $2r+2$.  This subsequence has $r+1$ positive and $r+1$ negative terms.

 Consider the positive terms, which come from $A$.  The subsequence uses at
 most one vertex from each consecutive pair $\{k_s,k_s+1\}$, since no support
 point lies strictly between these two indices.  To obtain $r+1$ positive
 terms, it must therefore use both free vertices $i,j$ and one vertex from
 each of the $r-1$ consecutive pairs.  The same conclusion holds for the two
 free vertices $i',j'$ of $B$ among the negative terms.

 Assume that $i\le i'$.  Equality is impossible because the same reduced
 support point cannot occur with both signs, so $i<i'$.  Since $i$ and $j$
 both occur as positive terms in an alternating sequence, a negative support
 point lies strictly between them.  As $i'$ is the first vertex of $B$, this
 gives $i'<j$.  Similarly, a positive support point lies strictly between
 $i'$ and $j'$.  If $j'\le j$, no point of $A$ can do this: $i<i'<j'\le j$
 and every vertex of $A$ other than $i,j$ lies after $j$.  This implies that $j<j'$, and
 $i<i'<j<j'$, as we wanted to show.
\end{proof}

We will use the following one-dimensional charging observation.

\begin{lemma}\label{lem:interval-charge}
 Let $\mathcal I$ be a family of intervals in $\rr$ with integer endpoints such that every pair properly crosses.  Then
 \[
 \sum_{(a,b)\in\mathcal I}\frac{1}{b-a}\le1.
 \]
\end{lemma}

\begin{proof}
 Choose an interval $(a,b)$ of minimum length $\ell=b-a$.  Every other
 interval has exactly one endpoint among the $\ell-1$ integers strictly
 between $a$ and $b$.  These interior endpoints are distinct.  Therefore
 $|\mathcal I|\le\ell$.  All intervals have length at least $\ell$, and the
 claimed inequality follows.
\end{proof}

Now we are ready to prove \cref{thm:even-main}.

\begin{proof}[Proof of \cref{thm:even-main}]
 Assign to a simplex $A$ of the form \eqref{eq:near-facet}, with free interval
 $(i,j)$, the weight
 \[
 w(A)=\frac{1}{j-i}.
 \]
 Let $Q$ be a blocking set and assign every selected simplex to one blocker in
 its relative interior.  By \cref{lem:crossing-intervals}, the free intervals
 of all simplices assigned to the same blocker are pairwise properly crossing.
 By \cref{lem:interval-charge}, the total weight assigned to one blocker is at
 most one.  Therefore,
 \[
 |Q|\ge\sum_A w(A),
 \]
 where the sum is over all sets of the form \eqref{eq:near-facet}.

 Fix the free length $\ell=j-i$.  A standard argument with separators shows that the number of selected simplices of free length $\ell$ is
 
 \[
 \binom{n-\ell-r+1}{r}.
 \]
 Therefore,
 \[
 |Q|\ge
 \sum_{\ell=1}^{n-2r+1}\frac{1}{\ell}
 \binom{n-\ell-r+1}{r}.
 \]
For fixed $r$, we have
 \[
 \binom{n-\ell-r+1}{r}
 =\frac{n^r}{r!}+O_r\bigl(n^{r-1}(\ell+1)\bigr).
 \]
 Dividing by $\ell$ and summing gives
 \[
 |Q|\ge\frac{n^r}{r!}\left(\sum_{\ell=1}^{n-2r+1}\frac{1}{\ell} \right)-O_r(n^r)
 =\frac{1}{r!}n^r\log n-O_r(n^r).
 \]
\end{proof}




\begin{corollary}\label{cor:even-excess}
 For fixed $r$ and $P$ as in \cref{thm:even-main}, we have
 \(\displaystyle
 \frac{b_{2r}(P)}{\tau_{2r}(P)}=\Omega_r(\log n)\).
\end{corollary}

\begin{proof}
 The cyclic $2r$-polytope has $\Theta_r(n^r)$ facets, so every triangulation
 already has $\Omega_r(n^r)$ codimension-one faces.  
 Conversely, glueing $T$ to a cone over its boundary gives a simplicial $2r$-sphere on $n+1$ vertices.
 
 The Upper Bound Theorem for simplicial spheres
 \cite{Stanley1975} gives $O_r(n^r)$ faces of dimension $2r-1$.  Therefore
 $\tau_{2r}(P)=\Theta_r(n^r)$, and the result follows from
 \cref{thm:even-main}.
\end{proof}

\section{Odd-dimensional moment curves}\label{sec:odd}

The proof for even dimension uses the special construction for simplices, which allows us to reduce the analysis to their free intervals.  In odd dimension, the argument does not generalize.  Moreover, for equally spaced points on the moment curve in odd
dimension, the simplices can instead be grouped into families with a common
interior point to give a set of blockers that matches the lower bound for the number of codimension-one faces in triangulations.

The lower bound comes from a standard regular triangulation of the cyclic
polytope; see \cite{DeLoera2010}.  We include a short polynomial
description because it also explains the count.

\begin{proof}[Proof of \cref{thm:odd-main}: lower bound]
 The case $n=2r$ is immediate, so we consider $n > 2r$.  Lift the points from $P=\{\gamma_{2r-1}(1), \dots, \gamma_{2r-1}(n)\}$ in $\gamma_{2r-1}$ in $\rr^{2r-1}$ to
 \[
 (t,t^2,\dots,t^{2r})\in\rr^{2r}.
 \]
 We call $P'$ the lift of $P$.  A subset of $2r$ points from the lifted moment curve is a lower facet of $P'$ precisely when the
 monic polynomial with these $2r$ roots is nonnegative at all the sampled
 parameters.  Since its sign is negative between the first and second root,
 between the third and fourth root, and so on, this happens exactly when the
 roots form $r$ disjoint pairs of consecutive integers.  Their projections
 are the maximal simplices of the lower regular triangulation.  The number of
 such matchings in a path on $n$ vertices is
 \[
 \binom{n-r}{r}.
 \]

 The boundary of the cyclic $(2r-1)$-polytope has
 \[
 2\binom{n-r}{r-1}
 \]
 facets, by the usual facet count for cyclic polytopes
 \cite{Ziegler1995}.  If $I$ is the number of internal codimension-one faces,
 double counting gives
 \[
 2r\binom{n-r}{r}=2I+2\binom{n-r}{r-1}.
 \]
 Therefore, the total number of codimension-one faces is
 \[
 I+2\binom{n-r}{r-1}
 =r\binom{n-r}{r}+\binom{n-r}{r-1}
 =(n-2r+2)\binom{n-r}{r-1}.
 \]
 This is a lower bound for $b_{2r-1}(P)$.
\end{proof}

\begin{proof}[Proof of \cref{thm:odd-main}: upper bound]
It is convenient to use homogeneous moment
vectors
\[
 v(t)=(1,t,\dots,t^{2r-1}).
\]

Recall that a polynomial $Q$ of degree at most $2r-1$ corresponds to a hyperplane containing points of $\gamma_{2r-1}$.  In particular, $Q(t) = 0$ if and only if $\gamma_{2r-1}(t)$ is contained in the corresponding hyperplane.  

We will construct a set of points that intersects the relative interiors of the simplices of  $P=\{\gamma_{2r-1}(1),\dots, \gamma_{2r-1}(n)\}$ of dimension $2r-2$.  To do so, we group those simplices into small sets.  We then use the polynomial correspondence with hyperplanes to construct a large affine space common to all the hyperplanes that contain simplices corresponding to a small set.

For every polynomial $f$ of degree at most $2r-1$, let $\ell_f$ denote the
linear functional satisfying $\ell_f(v(t))=f(t)$.

Every $(2r-1)$-subset can be written uniquely as
\begin{equation}\label{eq:odd-alternation}
 a_1<u_1<a_2<u_2<\cdots<u_{r-1}<a_r.
\end{equation}
We group these simplices by the $(r-1)$-tuple \( U=(u_1,\dots,u_{r-1})\) and by the sum \(s=a_1+\cdots+a_r\).  For example, with $r=2$ we would have the triples $(1,4,7)$, $(2,4,6)$, and $(3,4,5)$ in the same group.  These form three flat triangles that share the vertex $\gamma_3(4)$ and one part of a ray starting from that vertex.

Fix one group $X$ and put
\[
 R(t)=\prod_{j=1}^{r-1}(t-u_j).
\]
Let $L$ be the affine subspace of the hyperplane $x_0=1$ which is the intersection of the hyperplanes defined by 
\begin{equation}\label{eq:L-equations}
 \ell_{R(t)\cdot t^k}=0\quad(0\le k\le r-2),
 \qquad
 \ell_{R(t)\cdot (t^r-st^{r-1})}=0.
\end{equation}
The simplex \( C=\conv\{v(u_1),\dots,v(u_{r-1})\}\)
lies in $L$.  This is because $u_i$ is a root of every single polynomial involved.

Consider a member of the group, and write
\[
 A(t)=\prod_{i=1}^r(t-a_i)
 =t^r-st^{r-1}+\sum_{k=0}^{r-2}c_kt^k.
\]
The hyperplane of the original simplex is given by $R(t)A(t)$ (it has the $2r-1$ needed roots).  Due to the expression above, it contains
$L$.  We now find a point of $L$ in the relative interior of the smaller simplex $\conv\{\gamma_{2r-1}(a_1), \dots, \gamma_{2r-1}(a_r)\}$.  


We look for this point in the form
\[
z_A=\sum_{i=1}^r \lambda_i v(a_i),
\qquad
\lambda_i>0,
\qquad
\sum_{i=1}^r\lambda_i=1.
\]

Let us first determine what the coefficients \(\lambda_i\) must
satisfy. Recall that \(L\) is defined by
\(
\ell_{R(t)\cdot t^k}=0\) for all $0\le k\le r-2$
and \(\ell_{R(t)\cdot(t^r-st^{r-1})}=0\). Since \( \ell_f(v(t))=f(t)\),
the first set of equations becomes
\[
\sum_{i=1}^r
\lambda_i R(a_i)a_i^k=0
\qquad (0\le k\le r-2).
\tag{*}
\]
These equations also imply the remaining defining equation of \(L\).
What we mean by this is that, writing
\[
A(t)=\prod_{i=1}^r(t-a_i)
=t^r-st^{r-1}+\sum_{k=0}^{r-2}c_kt^k,
\]
the equality \(A(a_i)=0\) gives
\[
a_i^r-sa_i^{r-1}
=-\sum_{k=0}^{r-2}c_ka_i^k.
\]
Therefore, \((*)\) implies
\[
\sum_{i=1}^r
\lambda_iR(a_i)(a_i^r-sa_i^{r-1})=0.
\]

We therefore need to find positive coefficients summing to one that
satisfy \((*)\). To achieve this, we use Lagrange interpolation.  To simplify the formulas, we use $A'(a_i)$ instead of $\prod_{j \neq i}(a_i - a_j)$.  For a polynomial $p(t)$ of degree at most $r-1$, the
Lagrange interpolation formula gives
\[
p(t)=
\sum_{i=1}^r
p(a_i)\frac{A(t)}{(t-a_i)A'(a_i)}.
\]
Comparing the coefficients, we obtain
\[
\sum_{i=1}^r\frac{a_i^k}{A'(a_i)}=0 \qquad \mbox{for $k=0,\dots,r-2$ and} \qquad \sum_{i=1}^r\frac{a_i^{r-1}}{A'(a_i)}=1.
\tag{**}
\]
In particular, taking \(p(t)=t^k\) gives the identities needed in
\((*)\).  We choose coefficients for which
\(\lambda_iR(a_i)\) is proportional to \(1/A'(a_i)\). Define
\[
\widetilde\lambda_i
=
\frac{1}{R(a_i)A'(a_i)}
\]
and normalize by setting
\(
\lambda_i
=
{\widetilde\lambda_i}/\left(
{\sum_{j=1}^r\widetilde\lambda_j}\right)\).
The choice of an alternating sequence
\[
a_1<u_1<a_2<\cdots<u_{r-1}<a_r
\]
implies that \(
\operatorname{sgn}R(a_i)
=
\operatorname{sgn}A'(a_i)\),
and therefore \(\widetilde\lambda_i>0\) for every \(i\). This means that the
\(\lambda_i\)'s are positive and sum to one.

By \((**)\), for every \(0\le k\le r-2\),
\[
\sum_{i=1}^r\lambda_iR(a_i)a_i^k
=
\frac{1}{\sum_{j=1}^r\widetilde\lambda_j}
\sum_{i=1}^r\frac{a_i^k}{A'(a_i)}
=0.
\]
Finally, this means that
\(
z_A\in
L\cap
\operatorname{relint}
\operatorname{conv}\{v(a_1),\dots,v(a_r)\}\).


Consider $H$ the affine hull of $\{v(u_i): u_i \in U\}$.  Then, $H$ is a hyperplane in $L$.  We claim that all the points $z_A$ are on the same side of $H$.  To do this, consider the polynomial $R(t)\cdot t^{r-1}$ and denote by $\varphi = \ell_{R(t)\cdot t^{r-1}}$ its associated linear form.  We have $H = L \cap \{\varphi = 0\}$.  However, when we evaluate $z_A$ in $\varphi$ we have

\begin{align*}
    \varphi(z_A) & = \sum_{i=1}^r \lambda_i R(a_i) a_i^{r-1} = \sum_{i=1}^r \left(\frac{1}{R(a_i)A'(a_i)\left(\sum_{j=1}^r \tilde{\lambda_j}\right)}\right)R(a_i)a_i^{r-1} = \\ & =\frac{1}{\sum_{j=1}^r \tilde{\lambda_j}}\sum_{i=1}^r \frac{a_i^{r-1}}{A'(a_i)} = \frac{1}{\sum_{j=1}^r \tilde{\lambda_j}} > 0,
\end{align*}

where the last cancellation comes from $(**)$ with $p(t) = t^{r-1}$. 

Let $x\in\relint(C)$ be any point, such as its centroid.  The affine span of $C$ is a hyperplane in $L$.
Choose a vector $w$ in the direction space of $L$, orthogonal to
$\aff(C)$ and pointing to the side on which all the points $z_A$ lie.  We can choose $\varepsilon>0$ such that all the $(2r-2)$-simplices in the group contain $x + \varepsilon w$ in their relative interiors.  In other words, one blocker suffices for all simplices in $X$, defined by the pair $(U,s)$.

It remains to count the groups.  The tuple $U$ satisfies
\[
 2\le u_1,\qquad u_{j+1}\ge u_j+2,
 \qquad u_{r-1}\le n-1,
\]
so there are \( \binom{n-r}{r-1}\)
choices.  For fixed $U$, the allowed values of each $a_i$ form an integer
interval.  The sum of their widths is $n-2r+1$, so the possible sums $s$ form
an interval of exactly $n-2r+2$ integers.  This means that the number of groups is
\[
 (n-2r+2)\binom{n-r}{r-1}.
\]
This gives a blocking set of size at most
$(n-2r+2)\binom{n-r}{r-1}$, as we wanted to show.
\end{proof}

\section{Remarks}

It is surprising that in odd dimensions the natural extension of the blocking conjecture, i.e., the logarithmic excess over the number of faces in a triangulation, fails.  The author believes that the lower linear bounds are far from optimal.

Among efforts by the author to improve the constant $41/13$, the most promising one is to bound the number of points in general position needed to generate $3$ compatible empty $5$-holes.  It is possible that $21$ points are enough, which would improve the lower bound in the Blocking Conjecture to $b(n)\ge\left(3 + \frac{3}{19}-o(1)\right)n$.  Of course, any lower bound that does not rely on counting $5$-holes would be much more interesting. 

\subsection*{Use of AI disclosure}

ChatGPT 5.6 suggested the grouping of codimension-one simplices used in the proof of Theorem 3 and identified the resulting behavior in odd dimensions. The author also used it for literature review.  The author independently verified every output used in the manuscript and assumes full responsibility for the results presented.


\end{document}